\documentclass{amsart}
\usepackage[dvipdfmx]{graphicx}
\usepackage[dvipdfmx]{color}
\usepackage{xcolor}
\usepackage{verbatim}
\usepackage{amsmath}
\usepackage{amssymb, mathrsfs}
\usepackage{amsbsy}
\usepackage{amscd}
\usepackage{amsthm}
\usepackage{bm}

\newcommand{\N}{\mathbb{N}}
\newcommand{\R}{\mathbb{R}}
\newcommand{\Z}{\mathbb{Z}}
\newcommand{\OO}{\mathbb{O}}
\newcommand{\E}{\mathbb{E}}

\newcommand{\e}{\epsilon}

\newcommand{\sgn}{\mathop{\mathrm{sgn}}\nolimits}
\renewcommand{\P}{\mathbb{P}}

\newcommand{\I}{\mathcal{I}}

\newtheorem{prop}{Proposition}

\newtheorem{thm}{Theorem}
\newtheorem{lem}{Lemma}
\newtheorem{Def}{Definition}
\theoremstyle{definition}
\newtheorem{remark}{Remark}
\usepackage{color}
\definecolor{lilas}{RGB}{182, 102, 210}

\numberwithin{equation}{section}
\begin{document}
\title{Ergodicity of the number of infinite geodesics originating from zero}
\date{\today}
\author{Shuta Nakajima} 
\address[Shuta Nakajima]
{Research Institute in Mathematical Sciences, 
Kyoto University, Kyoto, Japan}
\email{njima@kurims.kyoto-u.ac.jp}


\begin{abstract}
  First-passage percolation is a random growth model which has a metric structure. An infinite geodesic is an infinite sequence whose all sub-sequences are shortest paths. One of the important quantity is the number of infinite geodesics originating from the origin.  When $d=2$ and an edge distribution is continuous, it is proved to be almost surely constant  [D. Ahlberg, C. Hoffman. Random coalescing geodesics in first-passage percolation]. In this paper, we will prove the same result for higher dimensions and general distributions.
\end{abstract}

\maketitle

\section{Introduction}
First-passage percolation was first introduced by Hammesley and Welsh
in 1965, as a model of fluid flow in random medium. In this model, we
consider the first passage time on $\Z^d$-lattice equipped with random
weights. A path is said to be optimal if it attains the first passage
time. Under weak conditions on distributions, the first passage times between two points define
have a metric structure. Therefore, optimal paths can be seen as geodesics and are central objects
of this model. An infinite geodesic is an infinite path of $\Z^d$
whose all sub-sequences are optimal paths. One of the important quantity is the number of infinite geodesics originating from the origin. It is expected to be infinity
and proved rigorously under un-proven limiting shape assumption when
the dimension is greater than or equal to $2$ in \cite{New95}. However, it is currently best known to be
at least $4$, which is shown in \cite{Hoff08}. See \cite{50,AH16} for more background and related works on infinite geodesics. The important property is that two infinite
geodesics tend to coalesce, which is called "coalescing
property". It is established in the case $d=2$ case for continuous distributions \cite{DH14,AH16}. This property allows us to use ergodic theory and Ahlberg and Hoffman showed
that the number of infinite geodesics originating from the origin is almost
surely constant \cite{AH16}. Note that their methods rely on the uniqueness of optimal paths between any two points, which follows from the continuity of the distribution, and special geometric properties of $\Z^2$-lattice, which for example allows one to define the counter-clockwise labeling of infinite geodesics. Our aim of this paper is to develop new techniques to establish the coalescing property for more general frameworks. And we will prove the above result both for general dimensions and distributions.\\

\subsection{Setting}
We consider the first-passage percolation on the lattice $\Z^d$ with $d\geq{}2$. The model is defined as follows. The vertices are the elements of $\Z^d$. Let us denote by $E^d$  the set of edges:
$$E^d=\{\{ v,w \}|~v,w\in\Z^d,~|v-w|_1=1\},$$
where we set $|v-w|_1=\sum^d_{i=1}|v_i-w_i|$ for $v=(v_1,\cdots,v_d)$, $w=(w_1,\cdots,w_d)$. Note that we consider non-oriented edge in this paper and we sometimes regard $\{ v,w \}$ as a subset of $\Z^d$ with a slight abuse of notation. We assign a non-negative random variable $\tau_e$ on each edge $e\in E^d$ as the passage time of $e$. The collection $\tau=\{\tau_e\}_{e\in E^d}$ is assumed to be independent and identically distributed with common distribution $F$. A path $\gamma$ is a finite sequence of vertices $(x_1,\cdots,x_l)\subset\Z^d$ such that for any $i\in\{1,\cdots,l-1\}$, $\{x_i,x_{i+1}\}\in E^d$. It is useful to regard a path as a subset of edges:
\begin{equation}\label{edge-vertex}
  \gamma=(\{x_{i},x_{i+1}\})^{l-1}_{i=1}.
\end{equation}
Without otherwise noted, we use this convention. Let us define the length of a path $\gamma$ as $\sharp \gamma=l-1$.\\
Given a path $\gamma$, we define the passage time of $\gamma$ as
$$t(\gamma)=\sum_{e\in\gamma}\tau_e.$$
 Given two vertices $v,w\in\R^d$, we define the {\em first passage time} between vertices $v$ and $w$ as
$$t(v,w)=\inf_{\gamma:v\to w}t(v,w),$$
 where the infimum was taken over all finite paths $\gamma$ starting at $v$ and ending at $w$. A path from $v$ to $w$ is said to be optimal if it attains the first passage time, i.e., $t(\gamma)=t(v,w)$. We denote by $\mathbb{O}(v,w)$ the set of all optimal paths from $v$ to $w$. If $F$ is continuous, i.e., $\P(\tau_e=a)=0$ for any $a\in\R$, when we fix starting and ending point, then an optimal path is uniquely determined. Then we still denote by $\OO(v,w)$ this optimal path with a slight abuse of notation.\\

 We say that an infinite sequence $(x_1,x_2\cdots)\subset\Z^d$ is an infinite geodesic if for any $1\le i<j$, $(x_i,\cdots,x_j)$ is an optimal path from $x_i$ to $x_j$. Denote by  $\mathcal{I}$ the set of all infinite geodesics and  $\mathcal{I}(v)$ the set of all infinite geodesics originating from $v$. Given two infinite geodesics $\Gamma_1$ and $\Gamma_2$, we say that they are distinct if $\sharp \{x\in\Z^d|~x\in \Gamma_1\cap \Gamma_2\}<\infty$, where we regard $\Gamma_1$ and $\Gamma_2$ as subsets of vertices in this definition. Otherwise, we say that $\Gamma_1$ and $\Gamma_2$ coalesce and write $\Gamma_1\sim\Gamma_2$. Let $\mathcal{N}=\mathcal{N}(\tau)\in\N\cup\{\infty\}$ be the number of  distinct infinite geodesics:
 $$\mathcal{N}=\max\{k\in\N\cup\{\infty\}|~\exists \Gamma_1,\cdots,\Gamma_k\in\mathcal{I}\text{  such that }\Gamma_i\not\sim \Gamma_j\text{ for any $i\neq j$}\}.$$
 We define the number of  distinct infinite geodesics originating from $v\in\Z^d$ as $$\mathcal{N}_v=\max\{k\in\N\cup\{\infty\}|~\exists \Gamma_1,\cdots,\Gamma_k\in\mathcal{I}(v)\text{  such that }\Gamma_i\not\sim \Gamma_j\text{ for any $i\neq j$.}\}.$$ Since $\mathcal{N}$ is invariant under lattice shift, by ergodicity, it is almost surely constant \cite{50}: there exists $N\in\N\cup\{\infty\}$ such that
 \begin{equation}\label{const-eq}
   \P(\mathcal{N}=N)=1.
   \end{equation}
 If $F$ is continuous, then since an optimal path is uniquely determined between any two vertices, it is easy to check that $\sim$ is an equivalence relation, $\mathcal{N}=\sharp [\mathcal{I}/\sim]$ and $\mathcal{N}_v=\sharp \mathcal{I}_v$.
 \subsection{Main results}
 \begin{Def}
  A distribution $F$ is said to be {\em useful} if  
   \begin{equation}\label{Def:useful}
     \P(\tau_e=F^-)<
    \begin{cases}
    p_c(d) & \text{if $F^-=0$} \\
    \vec{p}_c(d)& \text{otherwise},
    \end{cases}
    \end{equation}
   where $p_c(d)$ and $\vec{p}_c(d)$ stand for the critical probabilities for $d$-dimensional percolation and oriented percolation model, respectively and $F^-$ is the infimum of the support of $F$.
 \end{Def}
 Note that if $F$ is continuous, then $F$ is useful.\\
\begin{thm}\label{main-thm}
  Suppose that $F$ is useful and there exists $\alpha>0$ such that $\E \exp{(\alpha\tau_e)}<\infty$. Then the following holds almost surely: for any $v\in\Z^d$,
  \begin{equation}
  \mathcal{N}_v=\mathcal{N}. \label{main-eq}
\end{equation}
  In particular, by \eqref{const-eq},
   $$\mathcal{N}_v\text{ is almost surely constant.}$$
\end{thm}
\begin{remark}
  In the case $d=2$ with a continuous distribution, the above result was shown in \cite{AH16}.
\end{remark}
\subsection{Notation and terminology}
This subsection collects some notations and terminologies for the proof.
\begin{itemize}
  \item Given a path $\gamma=(x_i)_{i=1}^l$, we set $\gamma[i]=x_i$.
\item Given two paths $\gamma_1=(\gamma_1[i])_{i=1}^{l}$ and $\gamma_2=(\gamma_2[i])_{i=1}^{l'}$ with $\gamma_1[l]=\gamma_2[1]$, we denote the concatenated path by $\gamma_1\oplus\gamma_2$, i.e. $\gamma_1\oplus\gamma_2=(\gamma_1[1]\cdots,\gamma_1[l],\gamma_2[1],\cdots,\gamma_2[l'])$.
\item Given two paths $\gamma=(y_i)^l_{i=1}$ and $\Gamma=(x_i)^{L}_{i=1}$, we write $\gamma\sqsubset\Gamma$ if there exists $k$ such that $y_{i}=x_{k+i}$ for any $i\in\{1\cdots,l\}$. Then we say that $\gamma$ is a sub--path of $\Gamma.$
\item Given $x,y\in\R^d$, we define $d_{\infty}(x,y)=\max\{|x_i-y_i|~i=1,\cdots,d\}$. It is useful to extend the definition as
  $$d_{\infty}(A,B)=\inf\{d_{\infty}(x,y)|~x\in A,~y\in B\}\text{\hspace{4mm}for $A,B\subset \R^d$}.$$
  When $A=\{x\}$, we write $d_{\infty}(x,B)$.
  \item Given $x\in\R$, we denote by $\lfloor x\rfloor$ the greatest integer less than or equal to $x$.
\item Given a set $D\subset\Z^d$, let us define the outer boundary of $D$ as
  $$\partial^+ D=\{v\notin D|~\exists w\in D\text{ such that }|v-w|_1=1\}.$$
\item Let $F^-$ and $F^+$ be the infimum and supremum of the support of $F$, respectively:
  $$F^-=\inf\{\delta\ge 0|~\P(\tau_e<\delta)>0\},~F^+=\sup\{\delta\ge 0|~\P(\tau_e>\delta)>0\},$$
  where if $F$ is unbounded distribution, then we set $F^+=\infty$.
\item Given a finite path $\gamma_0$ starting at $v$, we define $\I(\gamma_0)=\{\Gamma\in\I(v)|~\gamma_0\sqsubset\Gamma\}$.
  \item Given $M\in\N$, let $\mathbb{T}_M$ be the set of all paths whose length is
$M$.
\end{itemize}
\section{Proof}
\begin{figure}[b]
  \includegraphics[width=4.0cm]{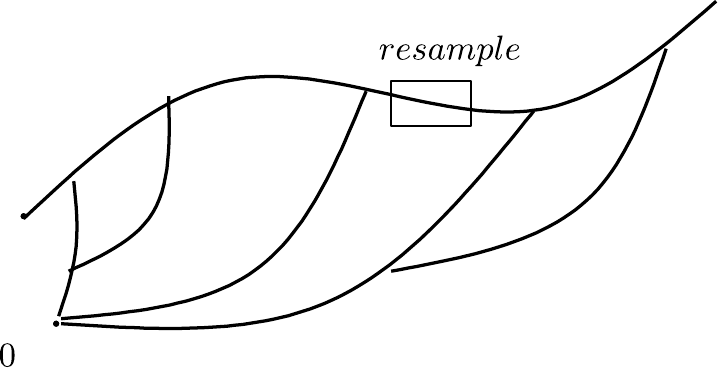}
  \hspace{8mm}
  \includegraphics[width=4.0cm]{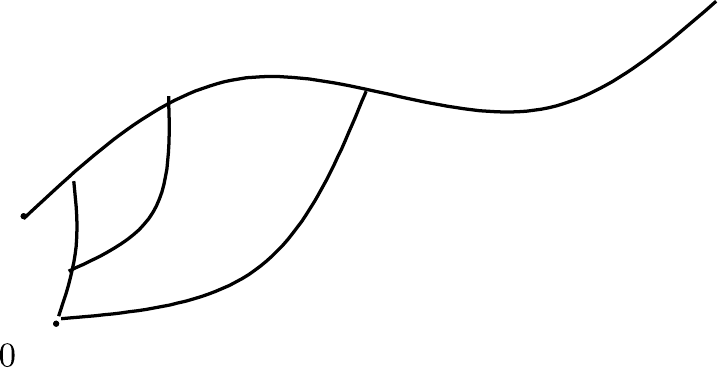}
\caption{}
Left: We resample the configurations on the third sub-path.\\
Right: After resampling, $4$-th and the subsequent bad points vanish.\\
  \label{fig:two}
\end{figure}
\subsection{Heuristic}
We will explain the heuristic behind the proof of $\mathcal{N}=\mathcal{N}_0$ in this subsection. Let $\Gamma$ be an infinite geodesic originating from some vertex with which all infinite geodesics originating from the origin do not coalesce. Then one can construct infinitely many optimal paths from the origin intersecting $\Gamma$ at only one point. We call the intersecting points bad points and the sub-path of $\Gamma$ between $k$-th bad point and $k+1$-th bad point the $k$-th sub-path. The crucial observation is the following: when we resample some configurations on the $k$-th sub-path and lower the passage time, the $k+1$-st and the subsequent bad points vanish and the $k$-th and the prior bad points remain. From this observation, one can expect 
$$\P(\text{$\Gamma$ has $k$ bad points})\geq c\P(\text{$\Gamma$ has infinitely many bad points}),$$
with some constant $c>0$. Here $c$ represents the cost for lowering the passage time. Since the events $\{\text{$\Gamma$ has $k$ bad points}\}_{k\in\N}$ are disjoint, summing up with respect to $k$, we have that for any $K\in\N$
$$1\geq \sum_{k=1}^K\P(\text{$\Gamma$ has $k$ bad points})\geq cK\P(\text{$\Gamma$ has infinitely bad points}).$$
Letting $K$ goes to infinity, we have
$$\P(\text{$\Gamma$ has infinitely many bad points})=0.$$
This implies that there exists $\Gamma_0\in\mathcal{I}(0)$ such that $\Gamma_0\sim\Gamma$ almost surely. We obatin $\mathcal{N}_0=N$. \\

There are mainly two obstacles to put the above argument into practice. First, to lower the passage times, each $k$-th paths needs to have sufficiently large passage time before resampling. Second, we need to take $\Gamma=\Gamma(\tau)$ depending on configurations. Then when we resample them, $\Gamma$ might change, i.e., $\Gamma(\tilde{\tau})\neq \Gamma(\tau)$ where $\tilde{\tau}=\{\tilde{\tau}_e\}_{e\in E^d}$ is resampled configurations. Therefore, the above heuristics does not work straightforwardly.
  \subsection{Proof for continuous distributions with unbounded support}\label{UBC}
  In this subsection, suppose that $F$ is continuous and $F^+=\infty$. Recall that we denote by $\OO(v,w)$ the unique optimal path between $v$ and $w$. It suffices to show that $\P(\mathcal{N}_0=N)=1.$
  \begin{Def}\label{badness}
    In this definition, we consider a path as a subset of vertices.
    Given an infinite path $\Gamma$ and a vertex $x\in\Gamma$, we say that $x$ is bad for $\Gamma$ if
    $$\mathbb{O}(0,x)\cap \Gamma=\{x\}.$$
    Otherwise, we say that $x$ is good for $\Gamma$.
      \end{Def}
  \begin{Def}
    Given an infinite path $\Gamma$, we say that $\Gamma$ is bad if
    $$\sharp \{i\in\N|~\Gamma[i]\text{ is bad for $\Gamma$ }\}=\infty.$$
Otherwise,  we say that $\Gamma$ is good.
  \end{Def}
   \begin{lem}\label{goodness}
If $\Gamma$ is good, then there exists an infinite geodesic $\Gamma_0\in\I(0)$ such that $\Gamma\sim \Gamma_0$. 
  \end{lem}
\begin{proof}
  Let $m=\max\{i\in\N|~\text{$\Gamma[i]$ is bad for $\Gamma$}\}$. It suffices to prove that
$\OO(0,\Gamma[m])\oplus (\Gamma[i])^{\infty}_{i=m}$ is an infinite geodesic.
We take $l\geq m$. Let $k=\min\{i\in\N|~\Gamma[i]\in\OO(0,\Gamma[l])\cap \Gamma\}.$ By the definition of $m$, we have $k\leq m$. Since $\OO(\Gamma[k],\Gamma[l])=(\Gamma[i])^l_{i=k}$, we have $\Gamma[m]\in\OO(0,\Gamma[l])$ and 
$$\OO(0,\Gamma[m])\oplus (\Gamma[i])^l_{i=m}=\OO(0,\Gamma[l]).$$
Since any sub-path of an optimal path is also an optimal path, we have that $\OO(0,\Gamma[m])\oplus (\Gamma[i])^{\infty}_{i=m}$ is an infinite geodesic.
\end{proof}
\begin{lem}\label{finite-abs}
If $\mathcal{N}_0<\mathcal{N}$, then there exists a finite path $\gamma_0=(\gamma_0[i])^l_{i=1}$ such that $\I(\gamma_0)$ is non-empty and for any $\Gamma'\in\I(\gamma_0)$, $\Gamma'$ is bad.
\end{lem}
\begin{proof}
  Since $\mathcal{N}_0<N$, there exists a bad infinite geodesic $\Gamma\in\I$. Note that $$\sharp \{\Gamma'\in\I(\Gamma[1])|~\Gamma'\text{ is good }\}\leq \mathcal{N}_0<N.$$
  Therefore, there exists $\ell\in\N$ such that for any $\Gamma'\in\I((\Gamma[i])^{\ell}_{i=1})$, $\Gamma'$ is bad.
\end{proof}
  This lemma yields
 \begin{equation*}
   \begin{split}
     \P(\I(0)<\mathcal{N})\le\sum_{\gamma_0}\P(\text{$\I(\gamma_0)$ is non-empty and $\forall \Gamma\in\I(\gamma_0)$, $\Gamma$ is bad}),
      \end{split}
 \end{equation*}
 where the summation is taken over all finite path. We fix a finite path $\gamma_0$ and set $v=\gamma_0[1]$. Let us define the event $\mathcal{A}$ as
 $$\mathcal{A}=\{\text{$\I(\gamma_0)$ is non-empty and $\forall \Gamma\in\I(\gamma_0)$ , $\Gamma$ is bad}\}.$$
 We will prove that for any finite path $\gamma_0$, $\P(\mathcal{A})=0$.\\
 
 Let $\e,M,L,\delta>0$.  We define the event $\mathcal{B}$ as
 \begin{equation}
   \mathcal{B}=\{t(0,v)\le M\}.
 \end{equation}Then if we take $M>0$ sufficiently large depending on $\e$, we get
  \begin{equation}\label{estimate-B}
 \P(\mathcal{B})\ge 1-\e/4.
 \end{equation}
 
  \begin{Def}
 Given $a,b\in\Z^d$, $(a,b)$ is said to be black if
   \begin{equation}
   \begin{cases}
     \begin{array}{l}
       |a-b|_1\ge \delta \sharp\OO(a,b),\\
       t(a,b)\geq \delta \sharp \OO(a,b),\\
       \sharp \{e\in \OO(a,b)|~\tau_{e}\ge 3M\}\ge  \delta \sharp \OO(a,b).\\
       \end{array}
   \end{cases}
   \end{equation}
 \end{Def}
 
  \begin{lem}\label{exp-decay}
    For any $M>0$, there exist $c_1,c_2>0$ and $\delta>0$ such that for any $k\in\N$,
    $$\P\left(\forall a,b\in[-k,k]^d,
    \begin{cases}
      (a,b)\text{ is black } &\text{ if $|a-b|_1\ge \sqrt{k}$}\\
    \sharp\OO(a,b)\leq k/2 &\text{otherwise}
      \end{cases}
\right)\le c_1\exp{(-c_2\sqrt{k})}.$$
  \end{lem}
  We postpone the proof until Appendix. The condition for $|a-b|_1<\sqrt{k}$ is necessary to restrict our attention to optimal paths whose length is sufficiently large, in order to use the condition that $(a,b)$ is black. We define the event $\mathcal{C}$ as
  
 $$\mathcal{C}=\left\{\forall k\ge L,~\forall a,b\in[-k,k]^d,
    \begin{cases}
      (a,b)\text{ is black } &\text{ if $|a-b|_1\ge \sqrt{k}$}\\
    \sharp\OO(a,b)\leq \delta k &\text{otherwise.}
      \end{cases}
\right\}$$

By Lemma~\ref{exp-decay}, we have the following lemma:
\begin{lem}
  For any $\e>0$ and $M>0$, there exist $\delta,L>0$ such that
 \begin{equation}\label{estimate-C}
   \begin{split}
     \P(\mathcal{C})\ge 1-\e/4
      \end{split}
 \end{equation}
 \end{lem}
   On the event $\mathcal{A}$, we take $\tilde{\Gamma}\in\I(\gamma_0)$ such that $\tilde{\Gamma}$ is bad with a deterministic rule.  We define the event $\mathcal{D}(a_1,\cdots,a_k)$ as
  $$\mathcal{D}(a_1,\cdots,a_k)=\{
\text{$\forall 1\le j\le k,~\exists i\in(a_{j-1},a_{j}]$ such that $\tilde{\Gamma}[i]$ is bad~for $\tilde{\Gamma}$}
  \}\cap \mathcal{A},$$
  with the convention that $a_0=1$. Note that $\mathcal{D}(a_1,\cdots,a_{k+1})\subset \mathcal{D}(a_1,\cdots,a_k)$ and $$\lim_{a_{k+1}\to\infty}\P(\mathcal{D}(a_1,\cdots,a_{k+1}))=\P(\mathcal{D}(a_1,\cdots,a_k)).$$
  Then we define the sequence $\{a_k\}_{k\in\N}$ inductively as follows:
  Let $a_1>2(\delta^{-1}M+L+|v|_1+1)$ be $\P(\mathcal{A}\backslash \mathcal{D}(a_1))<\e/4$. Suppose that we have defined $\{a_j\}^k_{j=1}$. We set $a_{k+1}$ such that $a_{k+1}>a_{k}$ and  
  $$\P(\mathcal{D}(a_1,\cdots,a_{k})\backslash \mathcal{D}(a_1,\cdots,a_{k+1}))<\e/2^{k+2}.$$
  We define
  \begin{equation*}
    \begin{split}
      \mathcal{D}=\left\{
  \exists \Gamma\in\I(\gamma_0) \text{ s.t. $\forall j\in\N$,  $\exists i\in(a_j,a_{j+1}]$ s.t. $\Gamma[i]$ is bad for $\Gamma$}
  \right\}.
    \end{split}
  \end{equation*}
  Note that $$\mathcal{D}\supset \bigcap_{k\in\N}\mathcal{D}(a_1,\cdots,a_k).$$
 Thus, we have
 \begin{equation}\label{estimate-D}
   \P(\mathcal{A}\backslash\mathcal{D})<\e/4.
   \end{equation}
 We define $\mathcal{P}=\mathcal{A}\cap \mathcal{B}\cap \mathcal{D}\cap\mathcal{C}.$ By \eqref{estimate-B}, \eqref{estimate-C} and \eqref{estimate-D}, for any $\e>0$, there exist $M,L,\delta>0$ such that $$\P(\mathcal{A})\leq \P(\mathcal{P})+\e.$$
 \begin{prop}\label{estimate-P}
   For any $M,L,\delta>0$
    $$\P(\mathcal{P})=0.$$
  \end{prop}
 Since $\e>0$ is arbitrary, this proposition leads to $\P(\mathcal{A})=0$ and we conclude the proof. Before going into the proof of Proposition~\ref{estimate-P}, we prepare some definitions.
  \begin{Def}
  We say that $\Gamma\in\I(v)$ has $k$--step if
  there exists $ i\in(a_k,a_{k+1}]$ such that $\Gamma[i]$ is bad and for any $ i>a_{k+1}$, $\Gamma[i]$ is good.
  \end{Def}
  \begin{Def}
    An edge $e\in E^d$ is said to be $k$-pivotal if 
      there exist $\Gamma\in\I(\gamma_0)$ and $i\in(a_k,a_{k+1}]$ such that $e=\{\Gamma[i-1],\Gamma[i]\}$ and for any $\Gamma'\in\I(\gamma_0)$ satisfying that $e\notin \Gamma'$ and $\Gamma'$ is good, there exists $j\in\N$ such that for any $ m\ge j$, $e\in \OO(0,\Gamma'[m])$. 
      \end{Def}

  \begin{Def}
    $\Gamma\in\I(v)$ is said to be very bad if for any $k\in\N,$ there exists $i\in(a_k,a_{k+1}]$ such that $\Gamma[i]$ is bad.
  \end{Def}
  \begin{Def}
    Given $\Gamma\in\I(v)$, let $\mathcal{S}(\Gamma)=\sup\{i\in\N|~\text{$\Gamma[i]$ is bad}\}$. If $\Gamma$ is bad, then we set $\mathcal{S}(\Gamma)=\infty$. Let  $\mathcal{R}=\inf\{\mathcal{S}(\Gamma)|~\Gamma\in\I(\gamma_0)\}$ and $\mathcal{K}=\inf \{t(0,x_{\mathcal{R}})|~\Gamma\in\I(\gamma_0)\text{ with }\mathcal{R}=\mathcal{S}(\Gamma)\}.$
    \end{Def}
  \begin{prop}\label{nonzero}
    For any $k\in\N$,
    $$\P(\delta a_{k-1}/2\leq \mathcal{K}<\infty)\ge \frac{\delta^2}{4}\P(\mathcal{P})\P(\tau_e<M).$$
  \end{prop}
  
 \begin{proof}[Proof of Proposition~\ref{estimate-P}]
  Since $\lim_{l\to\infty}\P(l\leq \mathcal{K}<\infty)=0$, letting $k\to\infty$, that is $a_{k-1}\to\infty$, we have $\P(\mathcal{P})=0$.
  \end{proof}
To prove Proposition~\ref{nonzero}, we will use the following lemma.

  \begin{lem}\label{cru}
     For any $k\geq 2$ and $\eta\in E^d$, 
      \begin{equation*}
        \begin{split}
          &\P\left(
        \{\eta\text{ is $k$-pivotal}\}\cap\{\delta a_{k-1}/2\leq \mathcal{K}<\infty\}\cap
          \{\sharp \{e\in E^d|~e\text{ is $k$-pivotal}\}\le 2\delta^{-1}a_{k+1}\}
      \right)\\
      &\ge \P\left(\left\{
      \begin{array}{c}
        \exists \Gamma\in\I(\gamma_0)\text{ s.t. $\Gamma$ is very bad},~\tau_\eta\ge 3M,\\
    \exists j\in(a_k,a_{k+1}]~s.t.~\eta=\{\Gamma[j-1],\Gamma[j]\}
        \end{array}
      \right\}\cap \mathcal{P}\right)\P(\tau_{\eta}<M)
          \end{split}
  \end{equation*}
  \end{lem}
 
  \begin{proof}
    Let $\{\tau^*_e\}_{e\in E^d}$ be independent copy of $\{\tau_e\}_{e\in E^d}$.Define $\{\tau^{(\eta)}_{e}\}_{e\in E^d}$ as
     $$\tau^{(\eta)}_{e}=  \begin{cases}
    \tau_{e}^* & \text{if $e=\eta$.}\\
  \tau_{e} & \text{if $e\neq \eta$.}
    \end{cases}$$
     We write that $\Gamma$ ia bad$^{(\eta)}$ if $\Gamma$ is bad with respect to $\tau^{(\eta)}$. We will use this convention for other properties. We have that the right hand side of \eqref{important} equals to
  \begin{equation}\label{important}
    \P\left(\left\{
      \begin{array}{c}
        \exists \Gamma\in\I(\gamma_0)\text{ s.t. },\text{$\Gamma$ is very bad},~\tau_\eta\ge 3M,\\
    \exists j\in(a_k,a_{k+1}]~s.t.~\eta=\{\Gamma[j-1],\Gamma[j]\}
        \end{array}
      \right\}\cap\mathcal{P}\cap\{\tau^{*}_{\eta}<M\}\right).
      \end{equation}
      We suppose the event inside of \eqref{important} and take such a path $\Gamma$ and $j$. It suffices to show that  
$$ \{\eta\text{ is $k$-pivotal$^{(\eta)}$}\}\cap\{\delta a_{k-1}/2\leq \mathcal{K}^{(\eta)}<\infty\}\cap \{\sharp \{e\in E^d|~e\text{ is $k$-pivotal$^{(\eta)}$}\}\le 2\delta^{-1}a_{k+1}\}.$$
      The proof is divided into five steps.
      
  \underline{Step 1}: $\Gamma\in\I^{(\eta)}(\gamma_0)$.
  \begin{proof}
    Note that for any $l$ with $l>j$, $(\Gamma[i])^l_{i=1}$ is an optimal path with respect to $\tau^{(\eta)}$. Since any sub--path of an optimal path is also optimal, we have $\Gamma\in\I^{(\eta)}(\gamma_0)$.
  \end{proof}

  \underline{Step 2}: $\Gamma$ has $k$--step or $k-1$--step with respect to $\tau^{(\eta)}$. In particular, $\mathcal{K}^{(\eta)}<\infty$.
      \begin{proof}
        Let $l>j$. Since $$M+t(0,\Gamma[l])\ge t(v,\Gamma[l])> t^{(\eta)}(v,\Gamma[l])+2M,$$
        we have $$t(0,\Gamma[l])> t^{(\eta)}(v,\Gamma[l])+M>t^{(\eta)}(0,\Gamma[l]).$$
        Thus $\eta\in\OO^{(\eta)}(0,\Gamma[l])$ and $\Gamma[l]$ is good$^{(\eta)}$.\\
        
  Next we take $l\le a_{k}$ such that $\Gamma[l]$ is bad for $\Gamma$. Then we will show that $\Gamma[l]$ is also bad$^{(\eta)}$ for $\Gamma$. In fact, if $\OO^{(\eta)}(0,\Gamma[l])\cap \Gamma\neq \{\Gamma[l]\}$, then $\eta\in\OO^{(\eta)}(0,\Gamma[l])$. Since $\Gamma$ is an infinite geodesic for $\tau^{(\eta)}$, there exists $l_1\ge j$ such that $\Gamma[l_1]$ is bad$^{(\eta)}$, which contradicts the above conclusion.
      \end{proof}

      \underline{Step 3}: For any good$^{(\eta)}$ $\Gamma_1\in\I^{(\eta)}(\gamma_0)$ with $\eta\notin \Gamma$ and for any sufficiently large $i\in\N$, we have
      $$\eta\in\OO^{(\eta)}(0,\Gamma_1[i])\text{ and }t^{(\eta)}(0,\Gamma_1[\mathcal{S}^{(\eta)}(\Gamma_1)])\ge \delta a_{k-1}/2.$$
  In particular, $\eta$ is $k$-pivotal$^{(\eta)}.$
    \begin{proof}
      By the same argument of Step 1, we get $\Gamma_1\in \I(\gamma_0)$. Thus by the condition of $\mathcal{A}$, $\Gamma_1$ is bad. We take $k_1\in\N$ so that $\Gamma_1$ has $k_1$--step for $\tau^{(\eta)}$. Let $l>a_{k_1+1}$ be such that $\Gamma_1[l]$ is bad for $\Gamma_1$. Then for any $l_1>l$, then since $\OO(0,\Gamma_1[l_1])\neq \OO^{(\eta)}(0,\Gamma_1[l_1])$, we have $\eta\in\OO^{(\eta)}(0,\Gamma_1[l_1])$. Since $\eta\notin \OO^{(\eta)}(\Gamma_1[\mathcal{S}^{(\eta)}(\Gamma_1)],\Gamma_1[l_1])\sqsubset \Gamma_1$, we obtain $\eta\in \OO^{(\eta)}(0,\Gamma_1[\mathcal{S}^{(\eta)}(\Gamma_1)])$. \\

      Recall that $\eta=\{\Gamma[j-1],\Gamma[j]\}$. Since  $\sharp \OO(v,\Gamma[j-1])= j-2\ge \sqrt{j+|v|_1}$, using the condition $\mathcal{C}$ with $k=j+|v|_1$, we have
   \begin{equation}\label{zettai}
     \begin{split}
       t^{(\eta)}(0,\Gamma_1[\mathcal{S}^{(\eta)}(\Gamma_1)])&\geq t^{(\eta)}(0,\Gamma[j-1])\\
       &\ge  t^{(\eta)}(v,\Gamma[j-1])- t^{(\eta)}(0,v)\\
       &=  t(v,\Gamma[j-1])- t^{(\eta)}(0,v)\\
       &\ge \delta a_{k}-M \ge \delta a_{k-1}/2.
        \end{split}
    \end{equation}
     \end{proof}

     \underline{Step 4}: 
     For any $\Gamma_1\in\I^{(\eta)}(\gamma_0)$ with $\eta\in\Gamma_1$ and $\mathcal{S}(\Gamma_1)<\infty$,
     $$t^{(\eta)}(0,\Gamma_1[\mathcal{S}^{(\eta)}(\Gamma_1)])\ge \delta a_{k-1}/2.$$
  \begin{proof}
    Since $(\Gamma[i])^j_{i=1}\sqsubset \Gamma_1\cap\Gamma$, $\Gamma_1$ has at least $k-1$--step. Therefore, by using the condition $\mathcal{C}$ with $k=a_{k-1}+|v|_1$, we have
     \begin{equation}
     \begin{split}
       t^{(\eta)}(0,\Gamma_1[\mathcal{S}^{(\eta)}(\Gamma_1)])&\ge t^{(\eta)}(v,\Gamma[a_{k-1}])-t^{(\eta)}(0,v)\\
       &\ge t(v,\Gamma[a_{k-1}])-M \ge \delta a_{k-1}/2
     \end{split}
     \end{equation}
     
        \end{proof}
   Combining Step 2-4, $\delta a_{k-1}/2\leq \mathcal{K}<\infty$ holds.\\ 

  \underline{Step 5}: If $e\in E^d$ is $k$-pivotal$^{(\eta)}$, then $e\in \OO^{(\eta)}(0,\Gamma[j])$ or $e\in \{\Gamma[1]\cdots,\Gamma[a_{k+1}]\}$. In particular,
  $$\sharp \{e\in E^d|~e\text{ is $k$-pivotal$^{(\eta)}$}\}\le 2\delta^{-1}a_{k+1}.$$ 
    \begin{proof}
      If $e\notin \Gamma$, then since $\Gamma$ is good$^{(\eta)}$ and $\gamma_0\sqsubset \Gamma$, there exists $l\ge j$ such that for any $m\ge l$, $e\in\OO^{(\eta)}(0,\Gamma[m])$. On the other hand, by Step 2, for any $m\ge l$, $\OO^{(\eta)}(\Gamma[j],\Gamma[m])\sqsubset \OO^{(\eta)}(0,\Gamma[m])$, which leads to $e\in \OO^{(\eta)}(0,\Gamma[j])$. If $e\in \Gamma$, then since there exists $\Gamma_1\in\I^{(\eta)}(v)$ and $i\in(a_k,a_{k+1}]$ such that $e=\{\Gamma_1[i-1],\Gamma_1[i]\}$, $e\in(\Gamma[i])^{a_{k+1}}_{i=1}$. Therefore $$e\in \OO^{(\eta)}(0,\Gamma[a_{k+1}])\cup \{\Gamma[1],\cdots,\Gamma[a_{k+1}]\}.$$

      Note that we have proved in Step 2 that $\OO^{(\eta)}(0,\Gamma[j-1])=\OO(0,\Gamma[j-1])$ and $\OO^{(\eta)}(0,\Gamma[j])=\OO^{(\eta)}(0,\Gamma[j-1])\oplus \eta$ . Thus, by  the condition $\mathcal{C}$, we obtain
      \begin{equation}
        \begin{split}
          \sharp \{e\in E^d|~e\in\OO^{(\eta)}(0,\Gamma[j])\}&\le \delta^{-1}|\Gamma[j]|_1\\
          &\leq \delta^{-1}(a_{k+1}+|v|_1).
    \end{split}
    \end{equation}
      Since $\sharp \{e\in E^d|~e\in(\Gamma[i])^{a_{k+1}}_{i=1}\}\leq a_{k+1}$, we have the conclusion.
      \end{proof}

    We turn to the proof of Lemma~\ref{cru}. By Step 1-5, we have
      \begin{equation*}
        \begin{split}
          &\P\left(
        \{\eta\text{ is $k$-pivotal}\}\cap \{\sharp \{e\in E^d|~e\text{ is $k$-pivotal}\}\le 2\delta^{-1}a_{k+1}\} \cap \{\delta a_{k-1}/2\leq \mathcal{K}<\infty\}
      \right)\\
      &=\P\left(\{\eta\text{ is $k$-pivotal$^{(\eta)}$}\}\cap \{\sharp \{e\in E^d|~e\text{ is $k$-pivotal$^{(\eta)}$}\}\le 2\delta^{-1}a_{k+1}\}\cap\{\delta a_{k-1}/2\leq \mathcal{K}^{(\eta)}<\infty\}
      \right)\\
      &\ge \P\left(\left\{
      \begin{array}{c}
        \exists \Gamma\in\I(\gamma_0)\text{ s.t. }, \text{$\Gamma$ is very bad},~\tau_\eta\ge 3M,\\
     \exists j\in(a_k,a_{k+1}] \text{ s.t. } \eta=\{\Gamma[j-1],\Gamma[j]\}
        \end{array}
      \right\},~\mathcal{P},~\tau^*_e<M\right),
          \end{split}
      \end{equation*}
      as desired.
  \end{proof}
\begin{proof}[Proof of Proposition \ref{nonzero}]  Note that if $\mathcal{C}$ holds and there exists $\Gamma\in\I(\gamma_0)$ such that  $\Gamma$ is very bad, then $$\sharp\{e\in E^d|~\tau_e\geq
  3M,~e\in(x_i)^{a_{k+1}}_{i=a_k}\}\geq \delta(a_{k+1}-a_k).$$
Therefore,
  \begin{equation*}
    \begin{split}
      &2\delta^{-1}a_{k+1}\P(\delta a_{k-1}/2\leq \mathcal{K}<\infty)\\
      &\ge \E\left[ \sharp\{e\in E^d|~e\text{ is $k$-pivotal}\};~
          \{\sharp \{e\in E^d|~e\text{ is $k$-pivotal}\}\le 2\delta^{-1}a_{k+1}\}\cap \{\delta a_{k-1}/2\leq \mathcal{K}<\infty\}
        \right]\\
      &=\sum_{e\in E^d}\P\left(
      \{\delta a_{k-1}/2\leq \mathcal{K}<\infty\}\cap \{e\text{ is $k$-pivotal}\}\cap    \{\sharp \{e\in E^d|~e\text{ is $k$-pivotal}\}\le (a_{k+1}-a_k)\}
        \right)\\
      &\ge \sum_{e\in E^d}\P\left(\left\{
      \begin{array}{c}
        \exists \Gamma\in\I(\gamma_0),\text{ $\Gamma$ is very bad}\\
        ~\tau_e\ge 3M,~e\in(\Gamma[i])^{a_{k+1}}_{i=a_k}
        \end{array}
      \right\},~\mathcal{P}\right)\P(\tau_e<M),\\
      &=\E\left[\sharp\{e\in E^d|~\exists \Gamma\in\I(\gamma_0) \text{ s.t. $\Gamma$ is very bad, $e\in(\Gamma[i])^{a_{k+1}}_{i=a_k}$, }\tau_e\geq
3M\};
      \mathcal{P}\right]\P(\tau_e<M),\\
      &\ge \delta(a_{k+1}-a_k)\P(\mathcal{P})\P(\tau_e<M)\ge
\delta \frac{a_{k+1}}{2}\P(\mathcal{P})\P(\tau_e<M).
    \end{split}
  \end{equation*}
\end{proof}
\subsection{Proof for continuous distributions with bounded support}\label{bounded}
Suppose that $F$ is continuous and $F^+<\infty$. The proof is similar as before, so we sketch the difference of them.  We take positive constants $\alpha_1,\alpha_2$ such that $F^-<\alpha_1<\alpha_2<F^+$. We replace the definitions of $\mathcal{B}$ and $\mathcal{C}$ as follows. Let us define the event $\mathcal{B}_2$ as
$$\mathcal{B}_2=\left\{t(0,v)\leq \frac{M}{3(\alpha_2-\alpha_1)}\right\}.$$
\begin{Def}
  $(a,b)\in\Z^d\times\Z^d$ is said to
be black2 if
   \begin{equation}
   \begin{cases}
     \begin{array}{l}
       |a-b|_1\ge \delta \sharp \OO(a,b),\\
       t(a,b)\geq \delta \sharp \OO(a,b),\\
       \sharp\{\gamma\in\mathbb{T}_M|~\gamma\sqsubset \OO(a,b),~\forall
e\in \gamma,~\tau_e\geq \alpha_2\} \ge
\delta \sharp \OO(a,b).\\
       \end{array}
   \end{cases}
   \end{equation}
\end{Def}

  \begin{lem}\label{exp-decay2}
    There exist $c_1,c_2>0$ such that for any $k\in\N$,
    $$\P\left(\forall a,b\in[-k,k]^d,
    \begin{cases}
      (a,b)\text{ is black2 } &\text{ if $|a-b|_1\ge \sqrt{k}$}\\
    \sharp\OO(a,b)\leq k/2 &\text{otherwise}
      \end{cases}
\right)\le c_1\exp{(-c_2\sqrt{k})}.$$
  \end{lem}
  We postpone the proof until Appendix. Then if we take $L$ sufficiently large, we have the following:
 \begin{equation}\label{happy2}
   \begin{split}
     \P\left(\forall k\ge L,~\forall a,b\in[-k,k]^d,
    \begin{cases}
      (a,b)\text{ is black } &\text{ if $|a-b|_1\ge \sqrt{k}$}\\
    \sharp\OO(a,b)\leq \delta k &\text{otherwise}
      \end{cases}
\right)\ge 1-\e/4
      \end{split}
      \end{equation}
 Let $\mathcal{C}_2$ be the event inside \eqref{happy2}. We define $\mathcal{P}_2=\mathcal{A}\cap\mathcal{B}_2\cap\mathcal{D}\cap\mathcal{C}_2$. Then as in subsection~\ref{UBC}, we have that for any $\e>0$, there exist $M,L,\delta>0$ such that
 $$\P(\mathcal{A})\leq\P(\mathcal{P}_2)+\e.$$
  \begin{Def}
   Given $\gamma=(\gamma_i)^l_{i=1}\in\mathbb{T}_M$, $\gamma$ is said
to be $k$-pivotal if there exists $\Gamma\in\I(\gamma_0)$ such that
$\gamma\sqsubset(\Gamma[i])^{a_{k+1}}_{i=a_k}$ and for any
$\Gamma'\in\I(\gamma_0)$ satisfying that $\gamma\cap \Gamma'=\emptyset$  and $\Gamma'$ is good, there exists
$j\in\N$ such that for any $ m\ge j$, $\gamma\cap
\OO(0,\Gamma'[m])\neq\emptyset$.
  \end{Def}
  Then Lemma~\ref{cru} will be replaced as follows:

  \begin{lem}\label{cru2}
     For any $k\geq 2$ and $\gamma_1\in\mathbb{T}_M$,
      \begin{equation*}
        \begin{split}
          &\P\left(
        \{\gamma_1\text{ is $k$-pivotal}\}\cap \{  \sharp \{\gamma\in\mathbb{T}_M|~\gamma\text{ is
$k$-pivotal}\}\le 2\delta^{-1} (2d)^M a_{k+1}\}\cap \{\delta a_{k-1}/2\leq \mathcal{K}<\infty\}
      \right)\\
      &\ge \P\left(\left\{
      \begin{array}{c}
        \exists \Gamma\in\I(\gamma_0)\text{ such that $\Gamma$ is very bad},\\
        ~\forall e\in
\gamma_1, \tau_e\geq \alpha_2,~\gamma_1\sqsubset(\Gamma[i])^{a_{k+1}}_{i=a_k}
        \end{array}
      \right\}\cap \mathcal{P}_2\right)\P(\tau_e<\alpha_1)^M
          \end{split}
  \end{equation*}
  \end{lem}
  \begin{proof}
    
    Let $\{\tau^*_e\}_{e\in E^d}$ be independent copy of $\{\tau_e\}_{e\in E^d}$. Define $\{\tau^{(\gamma_1)}_{e}\}_{e\in E^d}$ as
     $$\tau^{(\gamma_1)}_{e}=  \begin{cases}
    \tau_{e}^* & \text{if $e\in\gamma_1$.}\\
  \tau_{e} & \text{otherwise.}
    \end{cases}$$
    Then Step 1, 2, 4 can be proved in the same way as before. We replace Step 3 andStep 5 by\\

      \underline{Step 3'}: For any good$^{(\eta)}$ $\Gamma_1\in\I^{(\eta)}(\gamma_0)$ with $\gamma\cap \Gamma=\emptyset$ and for any sufficiently large $i\in\N$, we have
      $$\gamma\cap \OO^{(\eta)}(0,\Gamma_1[i])\neq \emptyset\text{ and }t^{(\eta)}(0,\Gamma_1[\mathcal{S}^{(\eta)}(\Gamma_1)])\ge \delta a_{k-1}/2.$$
      In particular, $\gamma$ is $k$-pivotal$^{(\eta)}.$
      
    \underline{Step 5'}: If $\gamma\in\mathbb{T}_M$ is
$k$-pivotal$^{(\gamma_1)}$, $\gamma\cap
\OO^{(\gamma_1)}(0,\Gamma[a_{k+1}])\neq\emptyset$ or $\gamma\cap
\{\Gamma[0]\cdots,\Gamma[a_{k+1}]\}\neq\emptyset$. In particular,
  $$\sharp \{\gamma\in\mathbb{T}_M|~\gamma\text{ is
$k$-pivotal$^{(\gamma_1)}$}\}\le 2\delta^{-1}(2d)^M a_{k+1}.$$
    They can be proved in the same way as in Lemma~\ref{cru}.
  \end{proof}
  
  \begin{prop}\label{nonzero2}
    For any $k\in\N$,
    $$\P(\delta a_{k-1}/2\leq \mathcal{K}<\infty)\ge \frac{\delta^2}{4(2d)^M}\P(\mathcal{P}_2)\P(\tau_e<M).$$
  \end{prop}
  \begin{proof}
      \begin{equation*}
    \begin{split}
      &2\delta^{-1} (2d)^M a_{k+1}\P(\delta a_{k-1}/2\leq \mathcal{K}<\infty)\\
      &\ge \E\left[ \sharp \{\gamma\in\mathbb{T}_M|~\gamma\text{ is
$k$-pivotal}\};~
         \{ \delta a_{k-1}/2\leq \mathcal{K}<\infty\}\cap
         \{ \sharp \{\gamma\in\mathbb{T}_M|~\gamma\text{ is
$k$-pivotal}\}\le 2\delta^{-1} (2d)^M a_{k+1}\}
        \right]\\
      &=\sum_{\gamma\in \mathbb{T}_d}\P\left(
      \{\gamma\text{ is $k$-pivotal}\}\cap  \{\delta a_{k-1}/2\leq \mathcal{K}<\infty\}\cap
          \{\sharp \{\gamma'\in\mathbb{T}_M|~\gamma'\text{ is
$k$-pivotal}\}\le 2\delta^{-1} (2d)^M a_{k+1}\}
          \right)\\
          &\ge \sum_{\gamma\in \mathbb{T}_d}\P\left(\left\{
      \begin{array}{c}
        \exists \Gamma\in\I(\gamma_0)\text{ such that }\\
        \text{$\Gamma$ is very bad},~\forall e\in \gamma_1,~\tau_e\geq \alpha_2,\\
        \gamma\sqsubset\{\Gamma[i]\}^{a_{k+1}}_{i=a_k}
        \end{array}
      \right\}\cap \mathcal{P}_2\right)\P(\tau_e<\alpha_1)^M\\
      &\ge \delta(a_{k+1}-a_k)\P(\mathcal{P}_2)\P(\tau_e<\alpha_1)^M\ge
      \delta \frac{a_{k+1}}{2}\cdot\P(\mathcal{P}_2)\P(\tau_e<\alpha_1)^M.
    \end{split}
  \end{equation*}
Rearranging it, we conclude the proof.
  \end{proof}
  Letting $k\to\infty$, we have $\P(\mathcal{P}_2)=0$. Finally, letting $\e\to 0$, we have $\P(\mathcal{A})=0$ as desired.
    
  \subsection{Proof for general distributions}
  We only consider the case $F^+=\infty$. For the case $F^+<\infty$ the proof is similar, combining the argument in subsection~\ref{bounded}. Let $K\in\N$.  We replace Definition~\ref{badness} as follows:
  \begin{Def}
    Given an infinite path $\Gamma$ and a vertex $x\in\Gamma$, we say that $x$ is bad for $\Gamma$ from $v\in\Z^d$ if there exists $\gamma\in \OO(v,x)$ such that
    $$\gamma\cap \Gamma=\{x\}.$$
    Otherwise, we say that $x$ is good for $\Gamma$ from $v$.
  \end{Def}
  \begin{Def}
    Given an infinite path $\Gamma$, we say that $\Gamma$ is bad from $v$ if
    $$\sharp \{i\in\N|~\Gamma[i]\text{ is bad for $\Gamma$ from $v$}\}=\infty.$$
Otherwise,  we say that $\Gamma$ is good from $v$.
  \end{Def}
  We simply say that $\Gamma$ is bad if $v=0$.  As in Lemma~\ref{goodness}, we have the following lemma.
  \begin{lem}\label{general-good}
    If $\Gamma$ is good, then there exists $\Gamma_0\in\I(0)$ such that $\sharp \Gamma\bigtriangleup\Gamma_0<\infty$, where $\bigtriangleup$ is symmetric difference and we regard $\Gamma$ and $\Gamma_0$ as subsets of vertices.
  \end{lem}
  Note that $\sharp \Gamma\bigtriangleup\Gamma_0<\infty$ is a stronger property than $\sharp \Gamma\sim\Gamma_0$.
\begin{Def}  
 Given $a,b\in\Z^d$, $(a,b)$ is said to be black3 if
   \begin{equation}
   \begin{cases}
     \begin{array}{l}
       |a-b|_1  \ge \delta \max_{\gamma\in\OO(a,b)}\sharp \gamma,\\
       t(a,b)\geq \delta \max_{\gamma\in\OO(a,b)}\sharp \gamma,\\
        \min_{\gamma\in\OO(a,b)} \sharp \{e\in \gamma|~\tau_{e}\ge 3M\}\ge  \delta \max_{\gamma\in\OO(a,b)}\sharp \OO(a,b).\\
       \end{array}
   \end{cases}
   \end{equation}   
Let us define the event as
$$\mathcal{C}_3=\left\{\forall k\ge L,~\forall a,b\in[-k,k]^d,
    \begin{cases}
      (a,b)\text{ is black3 } &\text{ if $|a-b|_1\ge \sqrt{k}$}\\
   \max_{_\gamma\in \OO(a,b)} \sharp \gamma\leq \delta k &\text{otherwise}
      \end{cases}
    \right\}$$
\end{Def}
 \begin{lem}
   For any $v\in\Z^d$ and $K\in\N$,
   $$\P(\{\mathcal{N}_v\leq K\}\cap \{\exists \Gamma\in \mathcal{I}(v)\text{ such that $\Gamma$ is bad}\})=0.$$
 \end{lem}
 \begin{proof}
   We follow the argument of subsection~\ref{UBC}. Let $\mathcal{A}_3=\{\mathcal{N}_v\leq K\}\cap \{\exists \Gamma\in \mathcal{I}(v)\text{ such that $\Gamma$ is bad.}\}$ and $\mathcal{P}_3=\mathcal{A}_3\cap \mathcal{B}\cap \mathcal{C}_3\cap \mathcal{D}$. An edge $e\in E^d$ is said to be $k$-{\rm pivotal}$_3$ if for any distinct infinite geodesics $\Gamma_1,\cdots,\Gamma_{K\land \mathcal{N}_v }\in\I(v)$, $$e\in\bigcup^{K\land \mathcal{N}_v}_{i=1}(\Gamma_i[i])^{\lfloor \delta^{-1}a_{k+1}\rfloor}_{i=1},$$
   Given an infinite geodesic $\Gamma\in\I(v)$, let $\tilde{\mathcal{S}}(\Gamma)=\sup\{\mathcal{S}(\Gamma')|~\Gamma'\in\I(v),~\Gamma\sim \Gamma'\}$. Let $$\mathcal{M}_k=\min_{\Gamma_1,\cdots,\Gamma_k}\max_{1\le i\le k}\tilde{\mathcal{S}}(\Gamma_i),$$
   where $\Gamma_1,\cdots,\Gamma_k$ run over all distinct $k$ infinite geodesics in $\I(v)$. We define the event $\mathcal{E}_3(k)$ as $$\mathcal{E}_3(k)=\{\{i\in\N|~a_{k-1}\le i\le a_{k+1}\}\cap \{\mathcal{M}_1,\cdots,\mathcal{M}_{K\land \mathcal{N}_v}\}\neq\emptyset\}.$$
   Note that $\lim_{k\to\infty}\P( \mathcal{E}_3(k))=0$.
   By definition, we have
   $$\sharp\{e\in E^d|~\text{$e$ is $k$-{\rm pivotal}$_3$}\}\leq K\delta^{-1}a_{k+1}.$$
   Let $\{\tau^*_e\}$ be independent copy of $\tau$ and we define $\tau^{(\eta)}$ as before. By the same argument as before, it suffices to prove the following: for any $\eta\in E^d$,
      \begin{equation}
        \begin{split}
          &\P(  \{\eta\text{ is $k$--{\rm pivotal}$^{(\eta)}_3$}\}\cap \mathcal{E}^{(\eta)}_3(k)) \\
          &\ge \P\left(\left\{
      \begin{array}{c}
        \exists \Gamma\in\I(v)\text{ s.t. $\Gamma$ is very bad},~\tau_\eta\ge 3M,\\
        \exists j\in(a_k,a_{k+1}]\text{ s.t. $\eta=\{\Gamma[j-1],\Gamma[j]\}$}
        \end{array}
      \right\}\cap \mathcal{P}_3\cap \{\tau^*_{\eta}<M\}\right).
          \end{split}
      \end{equation}
      To this end, suppose that the event inside of the right hand side holds.\\
      
      \begin{lem}\label{four}
        The following hold:
        
        \begin{enumerate}
      \item  for any $\Gamma'\in\mathcal{I}^{(\eta)}(v)$ with $\eta\notin \Gamma'$, $\Gamma'\in\mathcal{I}(v)$,
      \item for any $\Gamma'\in\I^{(\eta)}(v)$ with $\Gamma'\sim\Gamma$, $\eta\in \Gamma'$,
      \item for any $\Gamma'\in\I^{(\eta)}(v)$ with $\eta\in \Gamma'$, $\tilde{\mathcal{S}}(\Gamma')\ge a_{k-1}$,
        \item $\tilde{\mathcal{S}}^{(\eta)}(\Gamma)\leq a_{k+1}$,
      \item $\max\{l\in\N|~\Gamma_1,\cdots,\Gamma_l\in \I^{(\eta)}(v),~\text{s.t. $\eta\notin \Gamma_i$ and }\Gamma_i\not\sim\Gamma_j~\text{if $i\neq j$}\}\leq (K\land \mathcal{N}^{(\eta)}_v)-1$.
        \end{enumerate}
        \end{lem}
      \begin{proof}
        (i)-(iv) can be proved in a similar way as in Step 2 of Lemma~\ref{cru}. If $\eta\notin\Gamma_i\in \I^{(\eta)}(v)$, then we have that $\Gamma_i\in\I(v)$ and $\Gamma\not\sim\Gamma$. Therefore, we obtain (v).
        \end{proof}
 By (v), for any distinct infinite geodesics $\Gamma_1,\cdots,\Gamma_{K\land \mathcal{N}_v^{(\eta)} }\in\mathcal{I}^{(\eta)}(v)$, $\eta\in \cup^{K\land \mathcal{N}_v^{(\eta)}}_{i=1}\Gamma_i$. Since $|\Gamma[j]-v|_1\leq a_{k+1}$, by the condition $\mathcal{C}_3$,
      $$\max_{\gamma\in \OO(v,\Gamma[j])} \sharp\gamma\leq \delta^{-1}a_{k+1}.$$
 Therefore $\eta$ is $k$-{\rm pivotal}$_3^{(\eta)}$. Next we prove that $\mathcal{E}_3(k)$ holds.  Let $j\in\N$ be such that for any $i\le j$, $\mathcal{M}_i<a_{k-1}$ and for any $i>j$, $\mathcal{M}_i\ge a_{k-1}$. By (ii) and (v) in Lemma~\ref{four}, we get $j\le  (K\land \mathcal{N}^{(\eta)}_v)-1$. Thus, it suffices to show $\mathcal{M}_{j+1}\leq a_{k+1}$. Take distinct infinite geodesics $\Gamma_1,\cdots,\Gamma_j$ such that $\tilde{\mathcal{S}}(\Gamma_i)<a_{k-1}$. Then since $\eta\notin\Gamma_i$ and $\Gamma_i\not\sim\Gamma$ for any $i$  by (iii) in Lemma~\ref{four}, defining $\Gamma_{l+1}=\Gamma$, $(\Gamma_1,\cdots,\Gamma_{l+1})$ are distinct infinite geodesics.  Thus, by (iv) in Lemma~\ref{four}, we have $\mathcal{M}_{j+1}\leq a_{k+1}$.\\

 The rest of the proof is the same as before.  
 \end{proof}
 Letting $K$ goes to infinity, we have $$\P(\{\mathcal{N}_v<\infty\}\cap \{\exists \Gamma\in \mathcal{I}(v)\text{ such that $\Gamma$ is bad}\})=0.$$
  Exchanging the roles of $0$ and $v$, we have that
  \begin{equation}\label{Lett}
    \begin{split}
      &\P(\{\mathcal{N}_0<\infty\}\cap \{\exists v\in\Z^d\text{ and }\Gamma\in \mathcal{I}(0)\text{
        such that $\Gamma$ is bad from $v$}\})\\
      \leq &\sum_{v\in\Z^d}\P(\{\mathcal{N}_0<\infty\}\cap \{\exists\Gamma\in \mathcal{I}(0)\text{
        such that $\Gamma$ is bad from $v$}\})=0.
      \end{split}
    \end{equation}
   Next lemma corresponds to Lemma~\ref{finite-abs}.
   \begin{lem}
     \begin{equation}\label{Yatt}
   \P(\mathcal{N}_0<N)\leq \P(\exists\text{finite path $\gamma_0$ such that $\I(\gamma_0)\neq\emptyset$ and $\forall\Gamma'\in\I(\gamma_0)$, $\Gamma'$ is bad}).
   \end{equation}
 \end{lem}
 \begin{proof}
   By \eqref{Lett}, it suffices to show that if $\mathcal{N}_0<\infty$ and for any $\Gamma'\in \mathcal{I}(0)$ and $v\in\Z^d$, then $\Gamma'$ is good from $v$, then the event of the right hand side \eqref{Yatt} holds. Let $v\in\Z^d$ and $\Gamma\in\I(v)$ such that for any $\Gamma'\in\I(0)$, $\Gamma'\not\sim \Gamma$. Then by Lemma~\ref{general-good}, $\Gamma$ is bad. We take distinct infinite geodesics $\{\Gamma_i\}^{\mathcal{N}_0}_{i=1}\subset\I(0)$. For any $i\in\{1,\cdots,\mathcal{N}_0\}$, since $\Gamma_i$ is good from $v$, there exists $\ell_i\in\N$ such that for any $\Gamma'\in\I(v)$ with $(\Gamma[i])^{\ell_i}_{i=1}\sqsubset \Gamma'$, $\Gamma'\not\sim\Gamma_i$. Let $\ell=\max_{1\le i\le \mathcal{N}_0}\ell_i$ and $\gamma_0=(\Gamma[i])^{\ell}_{i=1}$. Note that for any $\Gamma'\in\mathcal{I}(v)$, if $\Gamma'\not\sim\Gamma_i$ for any $i$, then since $\mathcal{N}_0$ is the maximum number of distinct infinite geodesics, $\Gamma'$ is bad. Therefore, for any $\Gamma'\in\I(v)$ with $\gamma_0\sqsubset \Gamma'$, $\Gamma'$ is bad.   
 \end{proof}
  \begin{Def}
     An edge $e\in E^d$ is said to be $k$-{\rm pivotal}$_4$ if 
      there exists $\Gamma\in\I(\gamma_0)$ with  $e\in (\Gamma[i])_{i=a_k}^{a_{k+1}}$ and for any $\Gamma'\in\I(\gamma_0)$, if $e\notin \Gamma'$ and $\Gamma'$ is good, then for any sufficiently large $ m$, $$e\in \bigcap_{\gamma\in\OO(0,\Gamma'[m])}\gamma.$$ 
      \end{Def}
    \begin{Def}
$$\mathcal{K}_4=\min_{\Gamma\in \I(\gamma_0)}\min\{t(0,\Gamma[\mathcal{S}(\Gamma')])|~\Gamma'\in\I(\gamma_0),~\Gamma\sim\Gamma',~\mathcal{S}(\Gamma')=\tilde{\mathcal{S}}(\Gamma)\}.$$
      \end{Def}
    Fix $v\in\Z^d$ and a finite path $\gamma_0$ starting at $v$. Let us define
    $$\mathcal{A}_4=\{\text{$\I(\gamma_0)$ is non-empty and $\forall\Gamma\in\I(\gamma_0)$, $\Gamma$ is bad}\},$$
    and $\mathcal{P}_4=\mathcal{A}_4\cap\mathcal{B}\cap \mathcal{C}_3\cap \mathcal{D}$
 \begin{lem}\label{general-cur}
  Given $\eta\in E^d$, we define $\tau^{(\eta)}$ and the term `very bad' as before. Then,
      \begin{equation*}
        \begin{split}
      &\P\left(\{\eta\text{ is $k$--{\rm pivotal}$_4^{(\eta)}$}\}\cap \{\sharp \{e\in E^d|~e\text{ is $k$--{\rm pivotal}$^{(\eta)}$}\}\le 2\delta^{-1}a_{k+1}\}\cap\{\delta a_{k-1}/2\le \mathcal{K}_4<\infty\}
      \right)\\
      &\ge \P\left(\left\{
      \begin{array}{c}
        \exists \Gamma\in\I(\gamma_0), \text{$\Gamma$ is very bad},~\tau_\eta\ge 3M,\\
        ~\exists j\in(a_k,a_{k+1}]\text{ s.t. $\eta=\{\Gamma[j-1],\Gamma[j]\}$}
        \end{array}
      \right\},~\mathcal{P}_4,~\tau^*_e<M\right),
          \end{split}
      \end{equation*}
 \end{lem}
 \begin{proof}
   Step 2--5 will be replace by:\\

   \underline{Step 2''}: $\tilde{\mathcal{S}}^{(\eta)}(\Gamma)<\infty.$ In particular, $\mathcal{K}_4<\infty.$\\
   
   \underline{Step 3}'': For any good$^{(\eta)}$ $\Gamma_1\in \I^{(\eta)}(\gamma_0)$  with $\eta\notin \Gamma_1$, for any sufficiently large $i\in\N$,
   $$\eta\in\bigcap_{\gamma\in \OO^{(\eta)}(0,\Gamma_1[i])}\gamma\text{ and }t^{(\eta)}(0,\Gamma_1[\mathcal{S}^{(\eta)}(\Gamma_1)])\ge \delta a_{k-1}/2.$$
  In particular, $\eta$ is $k$--{\rm pivotal}$^{(\eta)}.$ \\

     \underline{Step 4''}: 
     The following hold:
     \begin{enumerate}
     \item for any $\Gamma_1\in\I^{(\eta)}(\gamma_0)$ with $\eta\in\Gamma_1$, there exists $\Gamma_2\in\I^{(\eta)}(\gamma_0)$ such that $(\Gamma[i])^{\infty}_{i=j-1} \sqsubset\Gamma_2$ and $ \mathcal{S}^{(\eta)}(\Gamma_2)\geq a_{k-1}$, in particular, $\tilde{\mathcal{S}}^{(\eta)}(\Gamma_1)\ge a_{k-1}$.
       \item if $\Gamma_2\in\I^{(\eta)}(\gamma_0)$ satisfying that $\mathcal{S}^{(\eta)}(\Gamma_2)\ge a_{k-1}$, then 
         $$t^{(\eta)}(0,\Gamma_2[\mathcal{S}^{(\eta)}(\Gamma_2)])\ge \delta a_{k-1}/2.$$
     \end{enumerate}
     
  \underline{Step 5}'': If $e\in E^d$ is $k$--{\rm pivotal}$^{(\eta)}$, then  $e\in \cap_{\gamma\in \OO^{(\eta)}(0,x_{j})}\gamma$ or $e\in \{\Gamma[1]\cdots,\Gamma[a_{k+1}]\}$. In particular,
  $$\sharp \{e\in E^d|~e\text{ is $k$--{\rm pivotal}$^{(\eta)}$}\}\le 2\delta^{-1}a_{k+1}.$$ 
  Except for Step 2'' and Step 3'', the proofs are the same as in Lemma~\ref{cru}.
  \begin{proof}[Proof of Step 2'']
    Let $\Gamma_1\in\I^{(\eta)}(\gamma_0)$ be such that $\Gamma_1\sim\Gamma$. Then by Lemma~\ref{four}-(i) and (ii), we obtain $\eta\in\Gamma_1$ and $\Gamma_1\in\mathcal{I}(\gamma_0)$. Let $l\in\N$  be such that $\Gamma_1[l]=\Gamma[j].$ By the same argument as in Step 2, we have  that for any $i\ge l$, $\Gamma_1[i]$ is good$^{(\eta)}$ for $\Gamma_1$. It follows that $ \tilde{\mathcal{S}}^{(\eta)}(\Gamma)<\infty.$
      \end{proof}
    \begin{proof}[Proof of Step 3'']
      By the same argument of Step 1, we get $\Gamma_1\in \I(\gamma_0)$. Thus by the condition of $\mathcal{A}_3$, $\Gamma_1$ is bad. Let $i\in\N$ be such that $i>\mathcal{S}^{(\eta)}(\Gamma_1)$ and $\Gamma_1[i]$ is bad. Since $t^{(\eta)}(0,\Gamma[i])<t(0,\Gamma[i])$, we have $\eta\in \bigcap_{\gamma\in\OO^{(\eta)}(0,\Gamma[i])}\gamma$. Let $l>i$ and $\gamma\in\OO^{(\eta)}(0,\Gamma_1[l])$. We define $l_1=\min\{l_2\in\N|~\gamma[l_2]\in\Gamma_1\}$ and let $l_1'$ be $\Gamma_1[l_1']=\gamma[l_1]$. Note that since $l_1'\le \mathcal{S}^{(\eta)}(\Gamma_1)$ and $(\gamma[1],\cdots\gamma[l_1])\oplus(\Gamma_1[l_1'],\cdots,\Gamma_1[l])\in\OO^{(\eta)}(0,\Gamma_1[l])$, we get $(\gamma[1],\cdots\gamma[l_1])\oplus(\Gamma_1[l_1'],\cdots,\Gamma_1[i])\in\OO^{(\eta)}(0,\Gamma_1[i])$. Together with $\eta\not\in(\Gamma_1[l_1'],\cdots,\Gamma_1[i])$, we have $\eta\in \gamma$.\\

      By the same argument as before, for any $\gamma\in\OO^{(\eta)}(0,\Gamma_1[\mathcal{S}^{(\eta)}(\Gamma_1)])$, we have
      $$\gamma\oplus(\Gamma_1[\mathcal{S}^{(\eta)}(\Gamma_1)],\cdots,\Gamma[i])\in \OO^{(\eta)}(0,\Gamma[i])\text{ and }\eta\in \bigcap_{\gamma\in\OO^{(\eta)}(0,\Gamma_1[\mathcal{S}^{(\eta)}(\Gamma_1)])}\gamma.$$
      With a similar argument to \eqref{zettai}, we obtain
   \begin{equation}
     \begin{split}
       t^{(\eta)}(0,\Gamma_1[\mathcal{S}^{(\eta)}(\Gamma_1)])&\geq  \delta a_{k-1}/2.
        \end{split}
    \end{equation}
    \end{proof}
    We now turn to the proof of Lemma~\ref{general-cur}. Let $\Gamma_1,\Gamma_2\in\I^{(\eta)}(\gamma_0)$ be such that $\Gamma_1\sim\Gamma_2$ and $\tilde{\mathcal{S}}^{(\eta)}(\Gamma_1)=\mathcal{S}^{(\eta)}(\Gamma_2)$. If $\eta\notin \Gamma_2$, then Step 3'' yields $t^{(\eta)}(0,\Gamma_2[\mathcal{S}^{(\eta)}(\Gamma_2)])\ge \delta a_{k-1}/2$. If $\eta\in \Gamma_2$, then by Step 4''-(i), there exists $\Gamma_3\in \I^{(\eta)}(\gamma_0)$ such that $(\Gamma[i])_{i=j-1}^{\infty}\sqsubset \Gamma_3$ and $\mathcal{S}^{(\eta)}(\Gamma_3)\ge a_{k-1}$. Then since $\Gamma_3\sim\Gamma_!$, we have $\tilde{\mathcal{S}}^{(\eta)}(\Gamma_1)\ge a_{k-1}$ and $t(0,\Gamma_2[\mathcal{S}^{(\eta)}(\Gamma_2)])\ge \delta a_{k-1}/2$ by usinig Step 4''-(ii). This yields $\delta a_{k-1}/2\le \mathcal{K}^{(\eta)}_4.$\\

    The rest of the proof is the same as before and we skip the details.
\end{proof}
  \section{Appendix} 
  \subsection{Proof of Lemma~\ref{exp-decay}}
\begin{lem}
  For any $M>0$, there exists $c,c_1,c_2>0$ such that for any $x\in\Z^d$,
  $$\P\left(\min_{\Gamma\in\OO(0,x)}\sharp\{e\in\Gamma|~\tau_e\ge M\}\leq c|x|_1\right)\leq c_1\exp{\{-c_2|x|_1\}}.$$
\end{lem}
\begin{proof}
We take $\tilde{\tau}_e$ such that if $\tau_e< M$, $\tilde{\tau}_e=\tau_e$ and otherwise, $\tilde{\tau}_e=\tau_e+1$. The results of \cite{BK93} imply that there exists $c>0$ such that for any $x\in\Z^d$,
$$\E[\tilde{t}(0,x)]\geq \E[t(0,x)]+c|x|_1.$$
Although they only discuss the first passage time from $0$ to $N\mathbf{x}_1$, the same proof works. By Theorem 3.11 in \cite{50}, we have that there exists $c_1,c_2>0$ such that
\begin{equation}
\begin{split}
  \P(|t(0,x)-\E[t(0,x)]|\geq c|x|_1/4)&\leq c_1\exp{(-c_2|x|_1)}\\
  \P(|\tilde{t}(0,x)-\E[\tilde{t}(0,x)]|\geq c|x|_1/4)&\leq c_1\exp{(-c_2|x|_1)}.
  \end{split}
\end{equation}
The yields that 
\begin{equation}
  \begin{split}
    \P(\tilde{t}(0,x)-t(0,x)\leq c|x|_1/2)\leq 2 c_1\exp{(-c_2|x|_1)}.
  \end{split}
\end{equation}
 Note that $\min_{\Gamma\in\OO(0,x)}\sharp\{e\in\Gamma|~\tau_e\ge M\}\leq c|x|_1/2$ implies $\tilde{t}(0,x)-t(0,x)\leq c|x|_1/2$. Therefore the proof is completed.
\end{proof}
\begin{lem}
  There exists $C,c_1,c_2>0$ such that
  \begin{equation}\label{length1}
    \P\left(\max_{\Gamma\in\OO(0,x)}\sharp \Gamma\geq C|x|_1\right)\leq c_1\exp{(-c_2|x|_1)}.\end{equation}
  There exists $c_1,c_2>0$ such that for any $k\in\N$ and $a,b\in \Z^d$ with $|a-b|_1\le \sqrt{k}$,
  \begin{equation}\label{length2}
    \P\left(\max_{\Gamma\in\OO(a,b)}\sharp \Gamma\ge k/2\right)\le c_1\exp{(-c_2|x|_1)}\end{equation}
\end{lem}
\begin{proof}
  From Proposition 5.8 in \cite{Kes86}, there exist $A,B,C>0$ such that for any $r>0$
   \begin{equation}\label{5.8}
     \begin{split}
       \P\left( \exists\text{ selfavoiding path $\Gamma$ from $0$
         with $|\Gamma|\geq r$ and $t(\Gamma) < A r$}\right) < B \exp{(-Cr)}.
     \end{split}
   \end{equation}
  We take a positive constant $C$ sufficiently large. We use Lemma~3.13 in \cite{50} and \eqref{5.8} with $r=C|x|_1$ to obtain,
   \begin{equation}\label{5.9}
     \begin{split}
       &\quad  \P\left(\max_{\Gamma\in\OO(0,x)}\sharp \Gamma\geq C|x|_1\right)\\
       &\leq
       \P\left( \exists\text{ self-avoiding path $\Gamma$ from $0$
         with $|\Gamma|\geq C|x|_1$ and $t(\Gamma) < A C|x|_1$}\right)+\P(t(0,x)\geq AC|x|_1)\\
       &\leq c_1\exp{(-c_2|x|_1)},
     \end{split}
   \end{equation}
   with some constant $c_1,c_2>0$.
   This yields \eqref{length1}.\\

   Note that $$\P\left(\max_{\Gamma\in\OO(a,b)}\sharp \Gamma\ge k/2\right)\leq    \P\left(\max_{\Gamma\in\OO(a,b)}\sharp \Gamma\ge k/2, t(a,b)< Ak/2\right)+   \P(t(a,b)\ge Ak/2).$$ The first term can be bounded by \eqref{5.8}. By exponential Markov inequality, the second term also can be bounded from above by $c_1e^{-c_2k}$ with some $c_1,c_2>0$. 
\end{proof}
\begin{proof}[Proof of Lemma~\ref{exp-decay}]
  If we take $\delta>0$ sufficiently small and $C,L>0$ sufficiently large,
  \begin{equation}
    \begin{split}
      \P(\mathcal{C}^c)&\leq \sum_{a,b\in[-k,k]^d,|a-b|_1\ge \sqrt{k}}\P((a,b)\text{ is not black })+\sum_{a,b\in[-k,k]^d,|a-b|_1< \sqrt{k}}\P(\max_{\Gamma\in\OO(a,b)}\sharp \Gamma\ge k/2)\\
      &\leq 3(2k+1)^d c_1e^{-c_2\sqrt{k}}\leq c_1e^{-c_2\sqrt{k}/2}.
  \end{split}
  \end{equation}
\end{proof}
\subsection{Proof of Lemma~\ref{exp-decay2}}
\begin{lem}\label{large M}
  Suppose that $F^+<\infty$ and $F$ is useful. Let $F^-<\alpha_2<F^+$. For any $M\in\N$ there exists
$c,c_1,c_2>0$ such that for any $x\in\Z^d$,
$$\P\left(\min_{\Gamma\in\OO(0,x)}\sharp\{\gamma\in\mathbb{T}_M|~\gamma\sqsubset \Gamma,~\forall
e\in \gamma,~\tau_e\geq \alpha_2\} \geq c
|x|_1\}\right)\leq c_1\exp{(-c_2 |x|_1)}.$$
\end{lem}
\begin{proof}
  Given a path $\Gamma$, we define the new passage time as
  $$t^+(\Gamma)=\sum_{e\in\Gamma}\tau_e+\beta \sharp\{\Gamma\in\mathbb{T}_M|~\gamma\sqsubset \Gamma,~\forall
e\in \gamma,~\tau_e\geq \alpha_2\},$$
  where $\beta$ is a positive constant chosen to be later. Let us
denote the corresponding first passage time from $0$ to $x$ by
$t^+(0,x)$.
\begin{lem}\label{Large M mean}
  There exists $c>0$ such that for any $x\in\Z^d$
  $$\E[t^+(0,x)]-\E[t(0,x)]\geq c|x|_1.$$
\end{lem}
First we prove Lemma~\ref{large M}. By Theorem~3.13 in \cite{50}, we have that
there exist $c_1,c_2$ such that
$$\P(|t(0,x)-\E[t(0,x)]|\geq c\beta|x|_1/4)\leq c_1\exp{(-c_2 |x|_1)}.$$
The same argument of \cite{DK14} leads to that
$$\P(t^+(0,x)-\E[t^+(0,x)]\leq -c\beta|x|_1/4)\leq
c_1\exp{(-c_2 |x|_1)}.$$
Therefore
$\P(t^+(0,x)-t(0,x)\leq c\beta|x|_1/2)\leq c_1\exp{(-c_2 |x|_1)}.$
Note that
$$\beta \min_{\Gamma\in\OO(0,x)}\sharp\{\gamma\in\mathbb{T}_M|~\gamma\sqsubset \Gamma,~\forall  e\in \gamma,~\tau_e\geq \alpha_2\}\geq t^+(0,x)-t(0,x).$$ Thus, we complete the proof.
\end{proof}
The proof of Lemma~\ref{exp-decay2} is the same as before. The rest will be devoted to Lemma~\ref{Large M mean}. Since
$\beta \min_{\Gamma\in\OO^+(0,x)}\sharp\{\gamma\in\mathbb{T}_M|~\gamma\sqsubset \Gamma,~\forall  e\in \gamma,~\tau_e\geq \alpha_2\}\leq t^+(0,x)-t(0,x)$, it suffices to show that
\begin{equation}\label{Goal}
  \E\left[\min_{\Gamma\in\OO^+(0,x)}\sharp\{\gamma\in\mathbb{T}_M|~\gamma\sqsubset \Gamma,~\forall  e\in \gamma,~\tau_e\geq \alpha_2\}\right ]\ge c|x|_1.
  \end{equation}
\begin{proof}[Proof of \eqref{Goal}]
  The proof is very similar to that of Lemma~5 in \cite{FPP-prop}. We only touch with the
  difference of them. Let $n\in\N$. We consider the following boxes:
$$S(l;n)=\{v\in\Z^d:nl\le v_i< n(l+1)\text{ for any $i$}\}.$$
$$T(l;n)=\{v\in\Z^d:nl-n\le v_i\le n(l+2)\text{ for any $i$}\}.$$
  $$B^j(l;n)=T(l;n)\cap{}T(l+2\sgn(j)\mathbf{e}_{|j|};n).$$

\begin{lem}\label{useful}
If $F$ is useful, then there exsits $\delta>0$ and $D>0$ such that for any $v,w\in\Z^d$,
$$\P(t(v,w)<\delta|v-w|_1)\leq{}e^{-D|v-w|_1}.$$
\end{lem}
For the proof of this lemma, see Lemma 5.5 in \cite{BK93}.

For simplicity, we set $B=B^j(l;n)$. We take
sufficiently large $R>0$ to be chosen later.

 \begin{Def} We define following conditions;\\

\noindent (1)for any $v,w\in B^j(l;n)$ with $|v-w|_1\geq{}n^{1/3}$,
$$t(v,w)\geq{}(F^{-}+\delta)|v-w|_1,$$
where $\delta>0$ is in Lemma~\ref{useful}. (Note that
$t^+(v,w)\geq{}t(v,w)$.)\\

\noindent(2)for any $e\cap B\neq\emptyset$, $\tau_e\leq{}F^+-R^{-1}$.\\

\noindent{}An $n$-box $B$ is said to be black if
$\begin{cases}
  \text{(1) and (2) hold}&\text{ if $\P(\tau_e=F^+)=0$}\\
  \text{(1) holds}&\text{ if $\P(\tau_e=F^+)>0$}\\
\end{cases}$
 \end{Def}
\end{proof}
 Hereafter ``crossing an $n$-box" means crossing in the short direction. See Figure~\ref{fig:two}.
 \begin{figure}[b]
  \includegraphics[width=4.0cm]{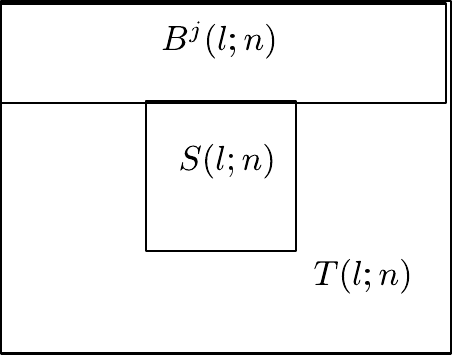}
  \hspace{8mm}
  \includegraphics[width=4.0cm]{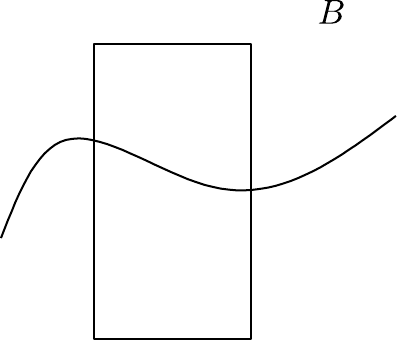}
\caption{}
Left: Boxes: $S$, $T$, $B$.\\
Right: $\mathbb{O}(0,x)$ crosses an $n$-box in the short direction.\\
  \label{fig:two}
\end{figure}
\begin{Def}
  ~\\
\noindent{}An $n$-box $B$ is said to be white if
there exists $\Gamma\in\mathbb{O}(0,x)$ such that $\Gamma$ cross $B$. \\

\noindent{}An $n$-box $B$ is said to be gray if $B$ is black and white.

 \end{Def}

As in (2.4) of \cite{FPP-prop}, we obtain that there exists $\e>0$ such that for any $x\in\Z^d$,
\begin{equation}\label{numb}
  \E[\sharp\{\text{distinct gray $n$-box $B$} \}]\geq{}\epsilon{}|x|_1/2
  \end{equation}

\begin{Def}
  Define\\
  
  \noindent$$F^+_R=\begin{cases}
  \text{$F^+-R^{-2}$}&\text{ if $F^+<\infty$ and $\P(\tau_e=F^+)=0$,}\\
  \text{$F^+$}&\text{ if $F^+<\infty$ and $F(\{F^+)>0$,}\\
  \end{cases}$$
  and\\
  $$F^-_R=\begin{cases}
  \text{$F^-+R^{-2}$}&\text{ if $ \P(\tau_e=F^-)=0$,}\\
  \text{$F^-$}&\text{ if $ \P(\tau_e=F^-)>0$.}
\end{cases}$$
\end{Def}
Note that if $R$ is sufficiently large,
\begin{equation}\label{estimate}
  F^-_R<F^-+\delta/2< F^+_R\text{ and }F^-_R\leq \alpha_2\leq F^+_R.
\end{equation}

     Denote by $\partial^+ B$ the outer boundary of an $n$-box $B$. Let $n_1=[d\sqrt{n}]+d$.
  If we take $n$ sufficiently large, for any $a,b\in\partial^+ B$ with \\
    $|a-b|_1\geq{}\delta n/(2F^+)$, there exists a self-avoiding path $\gamma_{a,b}=(x_0,\cdots,x_{l})$ from $a$ to $b$ satisfying $\{x_i\}^{l-1}_{i=1}\subset B$ such that the following hold:
   \begin{equation}\label{gam}
     \begin{array}{l}
       \text{ (1) $d_{\infty}(x_{n_1},B^c),d_{\infty}(x_{l-n_1},B^c)\geq \sqrt{n}$},\\
       \text{ (2) $|x_{n_1}-x_{l-n_1}|_1=l-2n_1$},\\
       \text{ (3) $d_{\infty}(x_i,B^c)\geq n_1$ for any $i\in\{n_1,\cdots,l-n_1\}$}.\\
       \text{ (4) $(x_{i-M},\cdots,x_{i+M})$ is a straight line for any $i\in I_{a,b}$},\\
     \end{array}    
   \end{equation}
   where   $I_{a,b}=\{n_1,\cdots,l-n_1\}\cap n_1\Z$.
   The reason why we use $\sqrt{n}$ is just $\sqrt{n}\ll n$ and not important. We take such a path to each $a,b\in\partial^+ B$ with $|a-b|_1\geq{}\delta n/(2F^+)$. For $a,b\in\partial^+ B$ with $|a-b|_1\geq{}\delta n/(2F^+)$, we take arbitrary self-avoiding path from $a$ to $b$.\\
   
   Let $a,b\in\partial^+ B$ with $|a-b|_1\geq{}\delta n/(2F^+)$ and $\gamma_{a,b}=(x_i)^l_{i=1}$. Given a  path $\gamma=\gamma_{a,b}=(x_0,\cdots,x_{|\gamma|})$ and $n$-Box $B$, $\tau$ is said to be satisfied $(\gamma,B)$-condition if (1) $\tau(x_{i-1},x_i)\in (\alpha_2,F^+_R]$ if there exists $j\in I_{a,b}$ such that $|i-j|\leq M$, (2) $\tau(x_{i-1},x_i)\leq F^-_R$ otherwise, (3) $\tau_e\geq{}F^+_R$ if $e\notin\gamma$ and $e\cap B\neq\emptyset$. Denote the independent copy of $\tau$ by $\tau^*$ and set $\tau^{B}$ as $\tau^{B}_e=\tau_e^{*}$ if $e\cap{}B\neq\emptyset$, $\tau^{B}_e=\tau_e$ otherwise. Let $(\tilde{a},\tilde{b})$ be random variable on $\partial^+ B\times \partial^+ B$ with uniform distribution and its probability measure $P$. Given a path $\Gamma=(x_0,\cdots,x_l)$ and an $n$-box $B$, we set
 $$\text{st}(\Gamma,B)=x_{\min\{i|~x_i\in \partial^+ B\}}\text{, fin}(\Gamma,B)=x_{\max\{i|~x_i\in \partial^+ B\}}.$$
   Note that if $\Gamma$ cross $B$ and $B$ is black, then since $\text{$t^+($st$(\Gamma,B)$,fin$(\Gamma,B))$}\geq{}(F^-+\delta )n$, $$\text{$|$st$(\Gamma,B)$-fin$(\Gamma,B)$}|_1\geq \frac{(F^-+\delta)n}{2F^+}+1.$$
   \begin{Def}
     An $n$-box $B$ is called $\mathbb{G}$ood if for any $\Gamma\in\mathbb{O}^+(0,x)$, there exists $\gamma\in  \mathbb{T}_M$ such that $\gamma\subset B$, $\gamma\sqsubset \Gamma$ and for any $e\in\gamma$, $\tau_e\geq \alpha_2$.
 \end{Def}

 \begin{lem}\label{ineq}
 We take $\beta=R^{-2}$. If $R\geq{}n^{2d}$ and $n$ is sufficiently large, then there exists $c>0$ such that for any $N\in\N$, unless $0\in B$ or $N\mathbf{e}_1\in B$,
 \begin{equation}\label{crucial2}
\begin{split}
   &\P(\text{ $B$ is $\mathbb{G}$ood for $\tau$})=P\otimes\P(\text{ $B$ is $\mathbb{G}$ood for $\tau^{B}$})\\
  & \geq{}P\otimes\P\left(
    \begin{array}{c}
      \text{$B$ is gray for $\tau$, $\exists\Gamma\in\mathbb{O}^+(0,x)$ s.t. $\Gamma$ cross $B$,} \\
      \text{$(\tilde{a},\tilde{b})=(\text{\rm{st}}(\Gamma,B),\text{\rm{fin}}(\Gamma,B)))$ and $\tau^*$ satisfies $(\gamma_{\tilde{a},\tilde{b}}$,$B)$-condition}\\
\end{array}
  \right)\\
  &=\frac{1}{|\partial^+{}B|^2}\sum_{(a,b)}\P\left(
  \begin{array}{c}
    \text{$B$ is gray for $\tau$, $\exists\Gamma\in\mathbb{O}^+(0,x)$ s.t. $\Gamma$ cross $B$,}\\
    \text{$(a,b)=(\text{\rm{st}}(\Gamma,B),\text{\rm{fin}}(\Gamma,B)))$, $\tau^*$ satisfies $(\gamma_{a,b}$,$B)$-cond.}\\
    \end{array}\right)\\
 &\geq{}c\P(\text{$B$ is gray}).\\
\end{split}
 \end{equation}
 \end{lem}
 \begin{proof}
   The proof is the same as in Lemma~5 of \cite{FPP-prop} and we skip the details.
 \end{proof}
 From \eqref{numb} and \eqref{crucial2}, we have that there exists $c>0$ such that  
\begin{equation}
\begin{split}
\E\left[\min_{\Gamma\in\OO^+(0,x)}\sharp\{\gamma\in\mathbb{T}_M|~\gamma\sqsubset \Gamma,~\forall  e\in \gamma,~\tau_e\geq \alpha_2\}\right ]&\geq{}\frac{1}{2d}\sum_{B^j(l;n):n\text{-box}}\P(B^j(l;n)\text{ is $\mathbb{G}$ood})\\
  &\geq \frac{c}{2d}\sum_{B^j(l;n):n\text{-box}}\P(B^j(l;n)\text{ is gray})\\
  &\geq \frac{c}{2d}\E[\sharp\{\text{distinct gray $n$-box $B$} \}]\\
  &\geq{}\frac{\e c}{4d}|x|_1,
  \end{split}
\end{equation}
where $2d$ appears because of the overlap of $n$-boxes. Thus the proof is completed.

\section*{Acknowledgements}
The author would like to express his gratitude to Christopher Hoffman and Daniel Ahlberg for helpful comments on the definition of the number of infinite geodesics for discrete distributions.  This research is supported by JSPS KAKENHI 16J04042.

\end{document}